\documentclass[11pt]{amsart}
\usepackage{graphicx,tikz,url}
\usepackage{amssymb,amscd}
\usepackage{amsfonts,mathrsfs}
\usepackage{setspace}
\usepackage{enumerate}
\usetikzlibrary{arrows, automata,chains,fit,shapes}
\usepackage{subfig,float}
\usetikzlibrary{decorations.markings}

\allowdisplaybreaks

\usepackage{color}

\newcommand{\ra}{\rightarrow}

\newcommand\supp{\mathrm{supp}}

\newcommand\Z{\mathbb Z}

\newcommand\R{\mathbb R}

\newtheorem{thm}{Theorem}[section]
\newtheorem{prop}[thm]{Proposition}
\newtheorem{lem}[thm]{Lemma}
\newtheorem{cor}[thm]{Corollary}
\newtheorem{conjecture}[thm]{Conjecture}

\theoremstyle{definition}
\newtheorem{defn}[thm]{Definition}

\newtheorem{example}[thm]{Example}

\begin{document}

\tikzset{->-/.style={decoration={
  markings,
  mark=at position #1 with {\arrow{>}}},postaction={decorate}}}

\title
{On a theorem of Avez}

\author{Murray Elder}
\address{School of Mathematical and Physical Sciences,
 University of Technology Sydney,
Ultimo NSW 2007, Australia}
\email{Murray.Elder@uts.edu.au}

\author{Cameron Rogers}
\address{Launceston Church Grammar School, Mowbray Heights, Tasmania 7248, Australia}
\email{cameron.m.rogers@gmail.com}

\date{\today}
\thanks{Research supported by Australian Research Council grant  FT110100178}

\begin{abstract}
For each symmetric, aperiodic  probability measure $\mu$ on 
a finitely generated group $G$, 
we define a subset 
$A_{\mu}$ consisting of  group elements $g$  for which the limit of the ratio 
${\mu^{\ast n}(g)}/{\mu^{\ast n}(e)}$ tends to $1$.
We prove that $A_\mu$ is
 a subgroup, is amenable, 
  contains every finite normal subgroup,
and  $G=A_\mu$ if and only if $G$ is amenable. 
  For non-amenable groups
we show that $A_\mu$ is not always a  normal subgroup, and can depend on the measure. We formulate some conjectures relating $A_\mu$ to the amenable radical. 
\end{abstract}

\maketitle

\section{introduction}
Let $\mu$ be a symmetric aperiodic probability measure $\mu$ on a finitely generated   group $G$ whose support generates $G$. Let $e$ denote the identity element of $G$ and $\mu^{\ast n}$ denote the $n$-fold convolution of the measure, so that $\mu^{\ast n}(g)$ is the probability that an $n$-step random walk induced by $\mu$ starting at $e$ ends at $g$. 
Avez \cite{Avez} showed that when $G$ is amenable, 
 \begin{equation}\label{eqn:avez}\lim_{n\ra \infty} \frac{\mu^{\ast n}(g)}{\mu^{\ast n}(e)}=1\end{equation}  for all $g\in G$.

In this paper we extend Avez' result in the following way. For an arbitrary finitely generated group $G$, 
we consider the   set, which we call $A_\mu$, of all $g\in G$ for which 
for which the limit of the ratio ${\mu^{\ast n}(g)}/{\mu^{\ast n}(e)}$ tends to $1$. Avez' result says if $G$ is amenable then $A_\mu=G$.
We prove that when $G$ is non-amenable, $A_\mu$ is a proper, amenable subgroup.
Moreover, $A_\mu$ contains every finite normal subgroup, so contains the  elliptic radical (the largest normal, locally finite subgroup of $G$), and so is non-trivial in many cases. We compute $A_\mu$ for some examples and show that in general it is not a normal subgroup and may depend on the measure.
We close by 
formulating some conjectures relating $A_\mu$ to the amenable radical.

 This work is part of PhD work of the second author \cite{CamPhD}, more details and applications can be found therein. Other relevant work which motivates the present paper 
  includes \cite{PropertyT, Erschler, 
Gournay, KaimV, PittetSaloffCoste, AnnesNotes, Woess}.

\section{Preliminaries}
In this article $\Z_+$ denotes the positive integers.
Recall that a probability measure $\mu$  on a group $G$ is {\em symmetric} if $\mu(x)=\mu(x^{-1})$ for all $x\in G$. The {\em support} of  $\mu$ is the set $\{x\in G\mid \mu(x)>0\}$ which we denote by $\supp(\mu)$.
The {\em convolution}, $\mu\ast \tau$, of two measures $\mu,\tau$ on a discrete group is  $\mu\ast\tau(y)=\sum_{x\in G} \mu(x)\tau(x^{-1}y).$
The distribution of a $n$-step random walk induced by $\mu$ is 
the $n$-fold convolution power of $\mu$, which we denote by $\mu^{\ast n}$.
The {\em period} of a measure $\mu$ is $\gcd\lbrace  n\in\Z_+\;|\;\mu^{\ast n}(e)>0\rbrace$.
The measure $\mu$ is said to be {\em aperiodic} if it has period 1.
Note that for a symmetric measure the period can only take values 1 or 2.

A function  $\zeta:G\ra \R$ on a finitely generated group $G$ is an {\em $\ell^2$-function}, or $\zeta\in \ell^2(G)$,  if $\sum_{g\in G} \left|\zeta(g)\right|^2$ is finite. The corresponding    inner product is $\langle \zeta,\iota \rangle_2=\sum_{g\in G} \zeta(g)\iota(g)$ and norm is $\left\Vert \zeta\right\Vert_2=\sqrt{ \langle\zeta,\zeta\rangle}$, as usual.
The
action of the group $G$ on $\ell^2(G)$ defined by
$g\cdot\zeta(x)=\zeta(g^{-1}x)$ for all $x\in G$ 
 is  called the {\em left regular  representation} of the group.

Observe that 
\begin{equation}\label{eqn:2ng}
\begin{split}
\mu^{\ast 2n}(g) = \sum_{x\in G} \mu^{\ast n}(x)\mu^{\ast n}(x^{-1}g)
= \sum_{x\in G} \mu^{\ast n}(x)\mu^{\ast n}(g^{-1}x)\\
= \sum_{x\in G} \mu^{\ast n}(x)\left(g\cdot \mu^{\ast n}(x)\right)
=\langle \mu^{\ast n},g\cdot \mu^{\ast n}\rangle\end{split}\end{equation}
and so $\mu^{\ast 2n}(e)=\langle \mu^{\ast n},\mu^{\ast n}\rangle=\left\Vert \mu^{\ast n}\right\Vert_2^2$.

The notion of amenability has many characterisations. Here we use the following.
\begin{thm}[\cite{Reiter, Day}] \label{thm:Reiter} $G$ is amenable if and only if there is a sequence $f_n$ of probability measures on $G$ such that  $\left\Vert g\cdot f_n-f_n \right\Vert_2\rightarrow 0$ for every $g\in G$.\end{thm}

\section{Defining $A_\mu$}

\begin{defn}\label{defn:Amu}
Let $G$ be a finitely generated group and let $\mu$ be a symmetric, aperiodic probability measure on $G$ whose support generates $G$. 
We define
\begin{align*}A_{G,\mu} &=\left\{g\in G \mid \lim_{n\ra \infty} \frac{\mu^{\ast n}(g)}{\mu^{\ast n}(e)}=1\right\}.\end{align*}   
When it is  understood which group 
 is being used the 
set will be referred to as $A_\mu$.
\end{defn}

The definition is clearly motivated by Avez' result: when $G$ is amenable we have $A_\mu=G$. 
A similar construction based on Theorem~\ref{thm:Reiter} would be the set  of all $g\in G$ for which $\left\Vert g\cdot f_n-f_n \right\Vert_2$ tends to 0 with respect to some fixed sequence $f_n$ of probability measures on $G$. 
An obvious choice for such a sequence would be $\xi_n=\frac{\mu^{\ast n}}{\Vert \mu^{\ast n} \Vert_2}$. It turns out that this construction  coincides with $A_\mu$.

\begin{prop}\label{prop:fromL2toMassAtVertex}
Let $G$ be a finitely generated group, and $\mu$ a symmetric, aperiodic probability measure on $G$ whose support generates $G$. 
Then
$$A_{\mu} =\left\{g\in G\mid \left\Vert g\cdot \xi_n-\xi_n\right\Vert_2\ra 0
\right\}.$$
\end{prop}

\begin{proof} 
By  Equation~(\ref{eqn:2ng}) we have 
\begin{align*}
\frac{\mu^{\ast 2n}(g)}{\mu^{\ast 2n}(e)} &=\frac{\langle \mu^{\ast n},g\cdot\mu^{\ast n}\rangle}{\left\Vert \mu^{\ast n}\right\Vert_2^2}
=\langle \xi_n,g\cdot\xi_n\rangle.
\end{align*}
Observe that
\begin{align*}\left\Vert g\cdot\xi_n-\xi_n\right\Vert_2^2 &=\sum_{x\in G} (g\cdot\xi_n-\xi_n)^2(x)\\
&=\sum_{x\in G} (g\cdot\xi_n)^2(x)-2\sum_{x\in G} (g\cdot\xi_n)(x)\xi_n(x)+\sum_{x\in G} (\xi_n)^2(x)\\
&=\left\Vert g\cdot\xi_n\right\Vert^2_2-2\langle g\cdot\xi_n,\xi_n\rangle+\left\Vert \xi_n\right\Vert_2^2\\
&=2-2\langle g\cdot\xi_n,\xi_n\rangle\end{align*}
since  $ \xi_n, g\cdot \xi_n$ are unit vectors.
Thus $\left\Vert g\cdot\xi_n-\xi_n\right\Vert_2$ approaches 0 if and only if  $\langle g\cdot\xi_n,\xi_n\rangle=\frac{\mu^{\ast 2n}(g)}{\mu^{\ast 2n}(e)} $ approaches 1. 
\end{proof}
\begin{cor}
$G$ is amenable if and only if $G=A_\mu$.
\end{cor}\begin{proof}Follows immediately from 
Theorem~\ref{thm:Reiter} and Proposition~\ref{prop:fromL2toMassAtVertex}.
\end{proof}

The following observation will be useful. \begin{lem}\label{lem:kpower} Let $\mu$ a symmetric, aperiodic probability measure on $G$ whose support generates $G$. 
For any fixed $k\in \Z_+$ we have $A_\mu=A_{\mu^{\ast k}}$.
\end{lem}

\section{Algebraic properties of $A_\mu$}

We now show that more than being some peculiar collection of elements, the sets $A_\mu$ have algebraic structure. Throughout this section we consider $G$  a finitely generated group and $\mu$  a symmetric, aperiodic probability measure on $G$ whose support generates $G$,  and 
$\xi_n=\frac{\mu^{\ast n}}{\Vert \mu^{\ast n} \Vert_2}$.

\begin{thm}\label{thm:subgroup}
 $A_\mu$ is a subgroup.
\end{thm}

\begin{proof}   
Let $g,h\in G$. We have 
\begin{align*}
\left\Vert gh\cdot \xi_n-\xi_n \right\Vert_2 
& =  \left\Vert g\cdot (h\cdot \xi_n-\xi_n) + g\cdot \xi_n-\xi_n \right\Vert_2 \\
&\leq
\left\Vert g\cdot (h\cdot \xi_n-\xi_n)\right\Vert_2 
+\left\Vert g\cdot \xi_n-\xi_n \right\Vert_2\\
&=
\left\Vert h\cdot \xi_n-\xi_n\right\Vert_2 
+\left\Vert g\cdot \xi_n-\xi_n \right\Vert_2
\end{align*}
since the $\ell^2$ norm is invariant under translation. Since $g,h\in A_\mu$ the right hand side limits to 0, so $gh\in A_\mu$. Clearly $e\in A_\mu$ and $A_\mu$ is closed under inverses since $\mu$ is symmetric.\end{proof}
In \cite{CamPhD} a slightly stronger statement is given, which gives some structural information about the cosets of $A_\mu$.

\begin{thm}\label{prop:amenablesubgroup}
$A_\mu$ is amenable.\end{thm}

The idea of our proof is to give a sequence of probability measures on $A_\mu$ which are almost invariant under the action of $A_\mu$. 
Proposition~\ref{prop:fromL2toMassAtVertex} says that we have such a sequence in $\ell^2(G)$, which we modify  to obtain a sequence in $\ell^2(A_\mu)$.

\begin{proof}
Choose a set $I=\lbrace s_1,s_2,\dots\rbrace$ of right coset representatives for $A_\mu$, which is  countable since $G$ is finitely generated. 
For $n\in\Z_+,s\in I$
define $\phi_{n,s}:G\ra \R$ by
$$\phi_{n,s}(x)=
\begin{cases}
\xi_n(x)\;\;\;\;\;\;\;\;\;\hfill\text{if $x\in A_\mu s$}\\
0\hfill\text{otherwise.}
\end{cases}$$
Then 
$$\xi_n=\sum_{s\in I}\phi_{n,s}.$$

Since $A_\mu$ is a subgroup, translation by $k^{-1}\in A_\mu$ on the left preserves the right cosets.  
Hence
$$k\cdot \xi_n=\sum_{s\in I}k\cdot \phi_{n,s}.$$

We will now construct a sequence of unit vectors in $\ell^2(A_\mu)$ which are almost invariant. 
For $n\in \Z_+, s\in I$ define $\psi_{n,s}:A_\mu\ra \R$ by $\psi_{n,s}(h)=\phi_{n,s}(hs)=
\xi_n(hs)$ where  $h\in A_\mu$.
Then $\psi_{n,s}\in\ell^2(A_\mu)$ since $\sum_{h\in A_\mu} \psi_{n,s}(h)=\sum_{h\in A_\mu} \xi_{n}(hs)$ is finite. We also have that the norm of $\psi_{n,s}$ in $\ell^2(A_\mu)$ is equal to the norm of $\phi_{n,s}$ in $\ell^2(G)$. We denote this norm by 
 $a_{n,s}$.  
 Note  $\sum_{s\in I}\left(a_{n,s}\right)^2=
\left\Vert \xi_n \right\Vert_2^2=1.$

Putting all this together:
\begin{align*}
\nonumber
\left\Vert k\cdot \xi_n-\xi_n\right\Vert_2^2 
& = \left\Vert \sum_{s\in I}k\cdot \phi_{n,s}-\sum_{s\in I}\phi_{n,s}\right\Vert_2^2\\
& = \left\Vert \sum_{s\in I}\left(k\cdot \phi_{n,s}-\phi_{n,s}\right)\right\Vert_2^2\\
& = \sum_{x\in G}\sum_{s\in I}
\left[\left(k\cdot \phi_{n,s}-\phi_{n,s}\right)(x)\right]^2\\
& = \sum_{s\in I}\sum_{x\in G}
\left[\left(k\cdot \phi_{n,s}-\phi_{n,s}\right)(x)\right]^2\\
& = \sum_{s\in I}\sum_{y\in A_\mu}
\left[\left(k\cdot \phi_{n,s}-\phi_{n,s}\right)(ys)\right]^2\intertext{\hfill(since $\phi_{n,s}$ is zero outside the $s$-coset)}
& = \sum_{s\in I}
\left(a_{n,s} \right)^2
\sum_{y\in A_\mu}
\left[\left(k\cdot \frac{\phi_{n,s}}{a_{n,s}}-\frac{\phi_{n,s}}{a_{n,s}}\right)(ys)\right]^2.\label{eqn:subgroup1}
\end{align*}
Now if  for all $s\in I$ 
 we have  
$$
\sum_{y\in A_\mu}
\left[\left(k\cdot \frac{\phi_{n,s}}{a_{n,s}}-\frac{\phi_{n,s}}{a_{n,s}}\right)(ys)\right]^2 \geq\epsilon
$$
then the above equation becomes
\begin{align*}
\left\Vert k\cdot \xi_n-\xi_n\right\Vert_2^2 
&\geq \epsilon\sum_{s\in I}
\left(a_{n,s} \right)^2=\epsilon.
\end{align*}
Therefore, $\left\Vert k\cdot \xi_n-\xi_n\right\Vert_2^2 < \epsilon$ implies there exists $s$ such that 
$$\sum_{y\in A_\mu}
\left[\left(k\cdot \frac{\phi_{n,s}}{a_{n,s}}-\frac{\phi_{n,s}}{a_{n,s}}\right)(ys)\right]^2 <\epsilon.
$$
Since $\left\Vert k\cdot \xi_n-\xi_n\right\Vert_2^2$
 limits to 
zero for every $k\in A_\mu$, there exists a sequence  $s_n$ for which 
\begin{equation}
\sum_{y\in A_\mu}
\left[\left(k\cdot \frac{\phi_{n,s_n}}{a_{n,s_n}}-\frac{\phi_{n,s_n}}{a_{n,s_n}}\right)(ys)\right]^2 \rightarrow 0.\label{eqn:subgroup2}
\end{equation}
for every $k\in A_\mu$.

Rewriting in terms of corresponding functions in $\ell^2(A_\mu)$, 
\begin{equation}
\nonumber
\sum_{y\in A_\mu}
\left[\left(k\cdot \frac{\psi_{n,s_n}}{a_{n,s_n}}-\frac{\psi_{n,s_n}}{a_{n,s_n}}\right)(y)\right]^2 
=\left\Vert k\cdot\frac{\psi_{n,s_n}}{a_{n,s_n}}-\frac{\psi_{n,s_n}}{a_{n,s_n}}\right\Vert_2^2
\rightarrow 0\label{eqn:subgroup2}
\end{equation}
and so  $\frac{\psi_{n,s_n}}{a_{n,s_n}}$
 supplies a sequence of almost invariant unit vectors in 
$\ell^2(A_\mu)$, and $A_\mu$ is amenable.\end{proof}

That $A_{\mu}$ is an amenable subgroup does not
preclude it being trivial for all
non-amenable $G$, nor does it guarantee
 that $A_{\mu}$ reflects any of the underlying 
structure of $G$. 
The next result shows that in many cases $A_\mu$ is an interesting non-trivial subgroup.

Recall that the {\em elliptic radical} of a finitely generated group $G$ is the largest normal, locally finite subgroup of $G$ (see for example \cite{Caprace}).  It is the group generated by all finite normal subgroups of $G$, and is contained in the amenable radical, the largest amenable normal subgroup.  We now prove a result which implies that  the elliptic radical 
 is contained in $A_\mu$.

\begin{thm}\label{thm:rell}
$A_\mu$ contains every finite normal subgroup of $G$. In particular, the elliptic radical 
 is contained in $A_\mu$.
\end{thm}
\begin{proof}
Let $F$ be a finite normal subgroup of $G$.  
Since the support of $\mu$  generates $G$, $F$ is finite and $\mu$ is aperiodic, we have  $\mu^{\ast |F|}(f)$ is non-zero for all $f\in F$.
Setting $\kappa=\mu^{\ast |F|}$ we have   $F\subseteq \supp(\kappa)$, and  
$A_\mu=A_\kappa$ by Lemma~\ref{lem:kpower}.

Let $S=\supp(\kappa)$. Each walk $(g_0=e, g_1, \dots)$ induced by $\kappa$ corresponds uniquely to a sequence $((h_0,f_0), (h_1,f_1),\dots)$ where $h_0=f_0=e$, $h_i\in\langle S\setminus F\rangle, f_i\in F$ and $g_i=h_if_i$ defined by the following process: if $g_n=g_{n-1}x, x\in \supp(\kappa)$ then 
$$(h_n,f_n)=\begin{cases}
(h_{n-1},f_{n-1}x)\hfill x\in F\\
(h_{n-1}x,x^{-1}f_{n-1}x)\;\;\;\;\;\;\;\;\;\hfill x\in S\setminus F.
\end{cases}$$

Define the measure $\phi:\langle S\setminus F\rangle\ra \R$ by  $$\phi(x)=
\begin{cases}
\kappa(F)\hfill x=e\\
\kappa(x)\;\;\;\;\;\;\;\;\;\hfill x\in S\setminus F
\end{cases}$$
Then $\phi^{\ast n}$ is the distribution of the first coordinate after $n$ steps.

The process on the second coordinate is a {Markov chain} on the state space $F$ where each move corresponds either to a right 
multiplication by $x\in F$ or to a conjugation by some element of $x\in S\setminus F$, each with probability $\kappa(x)$.
Let $\tau_n$ denote the distribution of the second coordinate after $n$ steps.  
We will prove that $\tau_n$ approaches the uniform distribution of $F$ using standard Markov chain theory (see for example \cite{ERvR} for further information).

Let $\Pr(f\ra g)$ be the probability of moving from state $f$ to $g$ in one step of the Markov chain.  If the step from $f$ to $g$ is a conjugation by $x$ (ie $g=x^{-1}fx$) then $f=xgx^{-1}$ so $\Pr(f\ra g)=Pr(g\ra f)$ since $\kappa$ is symmetric. Otherwise, the step is induced by right multiplication and clearly $\Pr(f\ra g)=\Pr(g\ra f)$  (since $g=fx$ only if $f=gx^{-1}$).
It follows that the Markov chain satisfies the {\em detailed balance} condition for the uniform measure $\pi=\frac1{|F|}$, ie $$\pi(f)\Pr(f\ra g)=\pi(g)\Pr(g\ra f)$$ and so 
 $\pi$ is a stationary distribution on $F$, that is $$\pi(f)=\sum_{y\in F} \pi(y)\Pr(y\ra f).$$

Since $e\in\supp(\kappa)$ the Markov process on $F$ is {\em aperiodic}, and since $F$ is finite the process is {\em irreducible}.
Then by the {Fundamental Theorem of Markov chains} 
(see for example \cite{ERvR} Theorem~3.12) 
$\tau_n$ converges to the unique stationary distribution $\pi$.

Now consider an $n$ step walk of the walk motivated by $\kappa$ which ends at some $f\in F$. We have 
$$\kappa^{\ast n}(f)=\sum_{g\in F}\phi^{\ast n}(g).\tau_n(g^{-1}f)$$ since to end at $f$ we must have first coordinate $g$ and second coordinate $g^{-1}f\in F$.
Then
\begin{align*}
\lim_{n\rightarrow\infty}\frac{\kappa^{\ast n}(f)}{\kappa^{\ast n}(e)} 
& =\lim_{n\rightarrow\infty}\frac{\sum_{g\in F} \phi^{\ast n}(g)\tau_n(g^{-1}f)}{\sum_{g\in F} \phi^{\ast n}(g)\tau_n(g^{-1})}\\
&=\lim_{n\rightarrow\infty}\frac{\sum_{g\in F} \phi^{\ast n}(g)\pi(g^{-1}f)}{\sum_{g\in F} \phi^{\ast n}(g)\pi(g^{-1})}=1\\
\end{align*} since $\pi$ is the uniform distribution on $F$, 
and so $F\subseteq A_\kappa=A_\mu$.
\end{proof}

Theorem \ref{thm:rell} is notable for two reasons.  
Firstly it shows that, whenever a finitely generated 
group contains a finite
normal subgroup, $A_{\mu}$ is non-trivial.  Secondly, this result
is independent of $\mu$.

\section{Examples}

Recall that non-abelian free groups have no non-trivial amenable normal subgroups.  That is, the amenable radical is trivial.

\begin{lem}\label{lem:freegroup}
Let $F_d$ be the free group of rank $d\geq 2$ with free basis generators including $a,b$, and let $\mu$ be a symmetric, aperiodic 
measure whose support generates $F_d$  satisfying $\mu(e)>0$ and $\mu(a)=\mu(b)>0$. Then $A_{F_d,\mu}$ is trivial.
\end{lem} 
\begin{proof}
Let $u\in\{a^{\pm 1}, b^{\pm 1}\}^+$. If $u\in A_\mu$ then  by interchanging $a^{\pm 1}$ with $b^{\pm 1}$ we obtain a word $v$ which also lies in $A_\mu$ by symmetry of the measure with respect to the generators $a,b$. If $u,v$ are not powers of the same element, in which case they generate a free group of rank 2, and since $A_\mu$ is an amenable subgroup it must be trivial. Otherwise if $u,v$ generate a cyclic group, choose instead to replace $a^{\pm 1}$ by $b^{\mp 1}$. 
\end{proof}

\begin{lem}\label{lem:productMeasure}
Suppose $G, H$ are finitely generated groups with 
symmetric, aperiodic probability 
 measures $\phi$ and $\psi$ respectively whose supports generate $G$ and $H$ respectively.  Recall that the product measure $\mu$ on $G\times H$ is defined by 
$$\mu(x,y)=\phi(x)\psi(y).$$ 
Then
 $$A_{G\times H, \mu}=A_{G,\phi}\times A_{H,\psi}.$$
\end{lem}\begin{proof}
To prove this we first note that $\mu^{\ast n}(x,y)=\phi^{\ast n}(x)\psi^{\ast n}(y)$.  This may be shown inductively.  It is true for $n=1$ by definition, and $\mu^{\ast n}(x,y)=\phi^{\ast n}(x)\psi^{\ast n}(y)$ implies

\begin{align*}
\mu_{n+1}(x,y) & =  
\sum_{(g,h)\in G\times H}\mu^{\ast n}(g,h)\mu(g^{-1}x,h^{-1}y)\\
& =  
\sum_{g\in G} \sum_{h\in H}
\left[
\phi^{\ast n}(g)\psi^{\ast n}(h)
\right]
\left[
\phi(g^{-1}x)\psi(h^{-1}y)
\right]\\
& =  
\sum_{g\in G} \sum_{h\in H}
\left[
\phi^{\ast n}(g)\phi(g^{-1}x)
\right]
\left[
\psi^{\ast n}(h)\psi(h^{-1}y)
\right]\\
& =  \sum_{g\in G} 
\phi^{\ast n}(g)\phi(g^{-1}x)
\sum_{h\in H}
\psi^{\ast n}(h)\psi(h^{-1}y)\\
& =  \phi_{n+1}(x)\psi_{n+1}(y).
\end{align*}

Thus
\begin{align*}
\lim_{n\rightarrow \infty}\frac{\mu^{\ast n}(g,h)}{\mu^{\ast n}(e_G,e_H)} 
& = 
\lim_{n\rightarrow\infty}\frac{\phi^{\ast n}(g)\psi^{\ast n}(h)}{\phi^{\ast n}(e_G)\psi^{\ast n}(e_H)}
= \lim_{n\rightarrow \infty} \frac{\phi^{\ast n}(g)}{\phi^{\ast n}(e_G)}\lim_{n\rightarrow \infty} \frac{\psi^{\ast n}(h)}{\psi^{\ast n}(e_H)}
\end{align*}
from which the result follows. 
\end{proof}

\begin{example}
Let $F_d$ be the free group of rank $d\geq 2$ with free basis generators including $a,b$, $\phi$ be a symmetric, aperiodic 
measure whose support generates $F_d$  satisfying $\phi(e)>0$ and $\phi(a)=\phi(b)>0$,  $H$  an amenable group with good  measure $\psi$, and $\mu$ be the product measure on $F_d\times H$.
Then $A_{F_d\times H,\mu}=H$, which is exactly the amenable radical of $F_d\times H$.
\end{example}

In light of these examples and the fact that $A_\mu$ contains the elliptic radical, one might ask whether $A_\mu$ is in fact always the amenable radical. If so, this would imply for one thing that the set $A_\mu$ is invariant under choice of  measure.
It turns out that this is not the case -- in the next section we  give an example where the amenable radical is trivial but $A_\mu$ is not. Moreover we show that $A_\mu$ depends on the choice of measure.

\section{Dependence on the measure}

\begin{prop}\label{prop:finite-inside-Amu}
Let $G$ be a finitely generated group with a finite subgroup $F$.  Then there exists a
 symmetric, aperiodic probability measure $\phi$ on $G$ whose support generates $G$
 such that $F\subset A_\phi$.
\end{prop}

\begin{proof}
Take $\psi=\pi_F\ast \mu\ast \pi_F$ where $\pi_F$ is the uniform measure on $F$.  
Then $$\phi(x)=\frac1{|F|^2}\sum_{f_1,f_2\in F} \mu(f_1xf_2)$$ which is symmetric, $\phi(e)\geq \frac1{|F|^2}\mu(e)>0$ and $\supp(\phi)\supseteq \supp(\mu)$. For $f\in F,x\in G$ we also have $\phi(fx)=\phi(x)$, so 
\begin{align*}
\psi^{\ast n}(f)
&=\sum_{g\in G}\psi(g)\psi^{\ast n-1}(g^{-1}f)\\
&=\sum_{x\in G}\psi(f^{-1}g)\psi^{\ast n-1}(g^{-1}f)\\
&=\psi^{\ast n}(e) 
\end{align*}
so  $F\subset A_\psi$.
\end{proof}

\begin{cor}
There exists a finitely generated group $G$ and symmetric, aperiodic probability measures $\mu,\tau$ on $G$ whose support generates $G$
 such that $A_{G,\mu}\neq A_{G,\tau}$.
\end{cor}

\begin{proof}
Consider the free product $G=\langle a\mid a^2=1\rangle\ast \langle b\mid b^3=1\rangle$.  By Proposition~\ref{prop:finite-inside-Amu} 
there are measures $\mu,\tau$ so that $a\in A_\mu$ and $b\in A_\tau$. If $A_\mu=A_\tau$ then $A_\mu=\langle a,b\rangle=G$ which is a contradiction since $G$ is not amenable.  Other examples are readily constructed from free products of finite groups.
\end{proof}

The same example also gives
\begin{cor}
There exists a finitely generated group $G$ and a symmetric, aperiodic probability measure $\mu$ on $G$ whose support generates $G$ so that $A_\mu$  is not equal to the amenable radical.
\end{cor}

\begin{proof}
Since $C_2\ast C_3$ contains finite subgroups,  we may use the arguments from Proposition \ref{prop:finite-inside-Amu} to construct
a measure $\mu$ for which $A_\mu$ is non-trivial.
However, $C_2\ast C_3$ has a trivial amenable radical.  This follows from the fact that it is $C^*$-simple \cite{c2StarC3_trivial_amenable_radical}, or by considering the action of the group on a tree. If $N$ is a normal amenable subgroup of a group acting on a tree, then by normality and the Tits alternative it fixes all vertices in $G/A$ (or $G/B$), or all edges, or a $G$-orbit of ends. Since the $G$-action on the space of ends is minimal, this implies in all 3 cases that $N$ is trivial.
\end{proof}

In particular, $A_\mu$ is not always  normal.

\section{Connection to the amenable radical}
In all cases considered, $A_\mu$ always {\em contains} the amenable radical. If this were true for all measures $\mu$,  the next results give a way to directly link the amenable radical with random walk distributions.

\begin{lem}
Let $G$ be a finitely generated group and $\mu$ a 
symmetric, aperiodic probability measure on $G$ whose support generates $G$. Define a measure $\mu_g:G\ra \R$ by 
 $\mu_g(x)=\mu(g^{-1}xg)$ for each $x\in G$.  Then  $\mu_g$ is a symmetric, aperiodic probability measure on $G$ whose support generates $G$, and 
$A_{\mu_g}=gA_\mu g^{-1}.$ 
\end{lem}

\begin{proof} 
We have $\mu_g(x^{-1})=\mu(g^{-1}x^{-1}g)=\mu(g^{-1}xg)=\mu_g(x)$, $\mu_g(e)=\mu(e)>0$ and $\supp(\mu_g)=g^{-1}\supp(\mu)g=G$.

For $y\in G$ 
\begin{align*}
{\mu_g}^{\ast 2}(y)
&=\sum_{x\in G}\mu_g(x)\mu_g(x^{-1}y)\\
&=\sum_{x\in G}\mu(g^{-1}xg)\mu(g^{-1}x^{-1}yg)\\
&=\sum_{h\in G}\mu(g^{-1}h)\mu(h^{-1}yg)\\
&=\mu^{\ast 2}(g^{-1}yg).
\end{align*}

Using an inductive argument it is clear that 
$${\mu_g}^{\ast n}(y)=\mu^{\ast n}(g^{-1}yg).$$

Now 
\begin{align*}
x\in A_{\mu^g}
&\iff \lim_{n\ra\infty}\frac{{\mu_g}^{\ast n}(x)}{{\mu_g}^{\ast n}(e)}=1\\
&\iff \lim_{n\ra\infty}\frac{\mu^{\ast n}(g^{-1}xg)}{\mu^{\ast n}(e)}=1\\
&\iff g^{-1}xg\in A_\mu\\
&\iff x\in gA_\mu g^{-1}
\end{align*}
\end{proof}

\begin{prop}\label{prop:outside-amen}
Let $G$ be a finitely generated group.  If $x\in G$ does not belong
to the amenable radical then  for any symmetric, aperiodic probability measure $\mu$ on $G$ whose support generates $G$  there exists $g\in G$ such that $x\notin A_{\mu_g}$.
\end{prop}
\begin{proof}
Suppose for contradiction  that $x\in G$ belongs to $A_\mu$ for every  symmetric, aperiodic probability measure $\mu$ on $G$ whose support generates $G$.  Then for some fixed $\mu$, by the previous lemma we have $x\in A_{\mu_g}=gA_\mu g^{-1}$ for all $g\in G$. Thus 
$$x\in\bigcap_{g\in G} g A_\mu g^{-1}$$ which is a normal amenable subgroup, hence $x$ belongs to the amenable radical.
\end{proof}

\begin{cor}
Let $\mathscr A_G$ denote the amenable radical of $G$, and $\mathscr M_G$ the set of all symmetric, aperiodic probability measures  on $G$ whose support generates $G$.
 If $\mathscr A_G\subseteq A_\mu$ for every  $\mu\in\mathscr M_G$, then $\mathscr A_G=\bigcap_\mu A_{\mu\in\mathscr M_G}.$
\end{cor}

We close by formulating two conjectures. \begin{conjecture}
Let $G$ be a finitely generated group.  Then for any symmetric, aperiodic probability measure $\mu$ on $G$ whose support generates $G$, the subgroup $A_\mu$ contains the amenable radical.
\end{conjecture}

Even more desirable would be the following.
\begin{conjecture}
Let $G$ be a finitely generated group.  Then there exists some $\mu$ such that $A_\mu$ is the amenable radical.
\end{conjecture}

\section*{acknowledgements}
The authors wish to thank Alain Valette, Sasha Fish,  Vadim Kaimanovich, Narutaka Ozawa, Anne Thomas  and the annonymous reviewer for helpful suggestions, references and feedback.

\bibliography{refs-abbrev} \bibliographystyle{plain}
\end{document}